%% file: ProcAMS_revision.tex
\documentclass[11pt]{amsart}
\usepackage{comment}
\usepackage[utf8]{inputenc}
\usepackage{enumitem}
\usepackage{colonequals,amsmath,amssymb,amsthm,mathtools,mathrsfs}
\usepackage[ruled,vlined]{algorithm2e}
\usepackage{float}
\usepackage{tikz,caption,subcaption}
\usepackage{pbox,tkz-graph}
\usetikzlibrary{positioning, shapes, plotmarks, decorations.markings, arrows,snakes,positioning}
\tikzstyle{vertex} = [circle, draw, fill, minimum size=4pt, inner sep=0pt]

\graphicspath{ {./images/} }
\usepackage[colorlinks=true]{hyperref}
\hypersetup{colorlinks=true,  linkcolor=blue,  filecolor=blue,  citecolor=blue, urlcolor=blue}
\usepackage{geometry}
\usepackage[capitalize,noabbrev]{cleveref}
\usepackage{pdfpages}
\usepackage{microtype}
\usepackage[justification=centering]{caption}
\usepackage{titlesec}
\usepackage{indentfirst}
\usepackage[style=ieee-alphabetic, maxbibnames=99, backend=biber,giveninits=true,noerroretextools=true,url=false,doi=true,isbn=false]{biblatex}
\newtheorem{theorem}[equation]{Theorem}
\newtheorem{proposition}[equation]{Proposition}
\newtheorem{lemma}[equation]{Lemma}
\newtheorem{corollary}[equation]{Corollary}

\theoremstyle{definition}
\newtheorem{definition}[equation]{Definition}
\newtheorem{remark}[equation]{Remark}
\newtheorem{example}[equation]{Example}
\newtheorem{question}[equation]{Question}

\numberwithin{equation}{section}

\newcommand{\an}{{\operatorname{an}}}

\newcommand{\can}{{\operatorname{can}}}
\newcommand{\Jac}{{\operatorname{Jac}}}
\newcommand{\Spec}{{\operatorname{Spec}}\,}

\newcommand{\red}{{\operatorname{red}}}
\newcommand{\wt}{{\operatorname{wt}}}
\newcommand{\val}{{\operatorname{val}}}
\newcommand{\coker}{\operatorname{coker}}

\newcommand{\Lie}{{\operatorname{Lie}}}

\newcommand{\ord}{{\operatorname{ord}}}
\newcommand{\vcan}{v_\mathrm{can}}
\renewcommand{\div}{\operatorname{div}\,}
\newcommand{\lra}{\longrightarrow}

\newcommand*\Z{\mathbb{Z}}
\newcommand*\Q{\mathbb{Q}}
\newcommand*\R{\mathbb{R}}
\newcommand{\A}{\mathscr{A}}
\newcommand*\C{\mathscr{C}}
\newcommand{\J}{\mathscr{J}}

\makeatletter
\renewcommand{\section}{\@startsection {section}{1}{\z@}
	{-3.5ex \@plus -1ex \@minus -.2ex}
	{2.3ex \@plus .2ex}
	{\Large \bfseries \filcenter}}
\makeatother

\begin{document}
\title{Jumps of Jacobians via orthogonal canonical forms}

\author{Enis Kaya, Michaël Maex and Art Waeterschoot}
\address{Department of Mathematics, KU Leuven, Celestijnenlaan 200B, 3001 Heverlee, Belgium.}
\thanks{EK was supported by FWO (grant GYN-D9843-G0B1721N). MM was supported by long term
structural funding (Methusalem grant) by the Flemish Government. AW was supported by FWO (grant 11F0123N)}

\email{enis.kaya@kuleuven.be}
\email{michael.maex@kuleuven.be}
\email{art.waeterschoot@kuleuven.be}

\begin{abstract}
Given a smooth, proper curve $C$ over a discretely valued field $k$, we equip the $k$-vector space $H^{0}(C,\omega_{C/k})$ with a canonical discrete valuation $\vcan$ which measures how canonical forms degenerate on regular integral models of $C$. More precisely, $\vcan$ maps a canonical form to the minimal value of its associated weight function, as introduced by Musta\c{t}\u{a}--Nicaise. Our main result states that $\vcan$ computes Edixhoven's jumps of the Jacobian of $C$ when evaluated in an orthogonal basis. As a byproduct, we deduce a short proof for the rationality of the jumps of Jacobians. We also show how $\vcan$ and the jumps can be computed efficiently for the class of $\Delta_v$-regular curves introduced by Dokchitser.
\end{abstract}

\maketitle

\section{Introduction} \label{section: intro}
\refstepcounter{equation}\subsection{Setup and motivation.}\label{motivation}

Suppose $k$ is a discretely valued field with ring of integers $R$, residue field $\tilde{k}$ and normalised valuation $v\colon k\twoheadrightarrow\overline{\Z}\coloneqq \Z\cup\{\infty\}$. Let $C$ be a smooth, proper and geometrically connected $k$-curve of genus $g\ge 1$. Write $\omega_{C/k}$ for the canonical sheaf of $C/k$ and $V\coloneqq H^0(C,\omega_{C/k})$ for the $k$-vector space of canonical forms on $C$; note that $\dim V=g$.

In this article, we study a canonical valuation $\vcan\colon V\to \overline{\Q}\coloneqq \Q\cup \{\infty\}$ which, roughly speaking, measures how canonical forms degenerate on regular integral $R$-models of $C$. Our main result, which is Theorem~\ref{thm:MainThmcopy} below, informally says that under the conditions that $k$ is Henselian with algebraically closed residue field and $C$ has index one\footnote{In other words, $C$ admits a zero divisor of degree one -- this is, for instance, the case if $C$ admits a $k$-rational point.}, the geometrically defined valuation $\vcan$ determines an interesting set of arithmetic invariants of $C$, namely Edixhoven's \emph{jumps} of the Jacobian variety $\mathrm{Jac}(C)$ of $C$ \cite{Edi92}. 
 
 As we recall in Section~\ref{subsection: jumps} below, if $k$ is a strictly Henselian discretely valued field, following Edixhoven one can define the jumps of a $g$-dimensional Abelian $k$-variety $A$ as a $g$-tuple of (not necessarily distinct) real numbers in $[0,1)$, which measure the behaviour of the N\'eron model of $A$ over tamely ramified extensions of $k$. In \cite[\S5.4.5]{Edi92}, Edixhoven asked if the jumps of $A$ are in fact rational. In general this question is an open problem, and we have an affirmative answer in the case of Jacobians by the work of Halle--Nicaise \cite[Corollary~6.3.1.5]{HN16} and Eriksson--Halle--Nicaise \cite[Theorem~4.4.1(a)]{EHN15}. Theorem~\ref{thm:MainThmcopy} provides a third argument by a new method, and we manage to make our approach explicit for the class of so-called $\Delta_v$-regular curves, as studied in \cite{Dok21}; see Theorem~\ref{cor:jumps delta-regular} and Example~\ref{example}. We remark that previously known methods for computing Edixhoven jumps include \cite{Edi92} (elliptic curves), \cite{Hal10} (elliptic curves and genus $2$~curves) and \cite{EHN15} (arbitrary curves); all of these methods require computation of the special fiber of a regular integral model with normal crossings.

\refstepcounter{equation}\subsection{Canonical valuation.}\label{sec:CanVal}
  More concretely, we can describe $\vcan$ as follows. Suppose $\C$ is a regular $R$-model of $C$, by this we mean that $\C/\mathrm{Spec}\,R$ is a regular proper flat $R$-curve with generic fiber isomorphic to $C/k$. Assume further that the reduced special fiber $\C_{s,\mathrm{red}}$ is a strict normal crossings divisor, in which case we refer to the model $\C$ as an \emph{snc} $R$-model of $C$. Suppose $E$ is a prime component of the reduced special fiber $\C_{s,\red}$. Then for $\omega\in V\setminus \{0\}$, we set
  \begin{equation}
   v_E(\omega)\coloneqq\frac{\ord_{E}(\div\omega)+1}{\mathrm{mult}(E)}-1\label{eqn:intro defn vE}
  \end{equation}
  where $\omega$ is viewed as a meromorphic section of the relative canonical sheaf $\omega_{\C/R}$. The extra additive terms $\pm1$ make this construction well behaved under tame base change and these have a log-geometric interpretation; see also Definition~\ref{defn: vcan} below. We set $v_E(0)\coloneqq\infty$ and $$\vcan(\omega)\coloneqq\min_{E}v_E(\omega),$$ where the minimum is taken over the prime components $E$ of $\C_{s,\red}$. In fact, the definition of $\vcan$ does not depend on the choice of $\C$; see Section~\ref{subsection: canonical valuation} for details. 

\refstepcounter{equation}\subsection{Properties of $\vcan$.}\label{intro: vcan is k-valuation} It is not hard to see that $\vcan$ is a \emph{discrete $k$-valuation on $V$}, by which we understand that 
\begin{enumerate}[label=(\alph*)]
    \item we have the ultrametric inequality $\vcan(\omega+\omega')\ge \min\{\vcan(\omega),\vcan(\omega')\}$ for all $\omega,\omega'\in V$;
    \item we have compatibility with the normalised valuation $v\colon k\twoheadrightarrow\Z$, i.e.,  $\vcan(a\omega)=v(a)+\vcan(\omega)$ for all $\omega\in V,a\in k$; and
        \item the image of $\vcan$ is contained in $\frac{1}{e}\overline{\Z}$ for some $e\in\Z_{>0}$. For instance, we can let $e$ be the least common multiple of the multiplicities of the prime components $E$. Beware that the image of $\vcan$ might not be closed under addition, so that it might fail to be of the form $\frac{1}{e}\overline{\Z}$ for some integer $e$.
\end{enumerate}
\begin{remark}
We stress that our definition of $\vcan$ is directly inspired by the weight functions of Musta\c{t}\u{a}--Nicaise \cite{MN15}, see Section~\ref{subsection: weight functions} and Lemma~\ref{lemma: defn vcan} below.  

We were also inspired from Temkin's study of metrics on $\omega_{C/k}$ \cite{T16}, as every sensible metric on $\omega_{C/k}$ (like the family of weight functions) induces a $k$-valuation on $V$, obtained as a minimum of $k$-valuations by evaluating in sections. 

With some imagination, one can also view the calculus of log canonical thresholds in singularity theory as another antecedent of $\vcan$.
\end{remark}

\refstepcounter{equation}\subsection{Main result.} For $\rho\in \R$, we set $V^{\ge \rho} \coloneqq \{\omega\in V:\ \vcan(\omega)\ge \rho\}$ and similarly we define $V^{>\rho}$. By Section~\ref{intro: vcan is k-valuation}, this induces a filtration on $V$ by $R$-lattices, and we denote by $\mathrm{Gr}^\rho V\coloneqq V^{\ge \rho}/V^{>\rho}$ the $\rho^{\text{th}}$ graded piece. Then the following theorem allows to compute the $\tilde{k}$-dimensions of these graded pieces.

\begin{theorem}\label{thm:MainThmcopy} Let $k$ be a Henselian discretely valued field with algebraically closed residue field $\tilde{k}$ and let $C$ be a smooth, proper and geometrically connected $k$-curve of genus $g\ge 1$ admitting a divisor of degree $1$.
Then for all $j\in [0,1)$, the integer $\dim_{\tilde{k}}\mathrm{Gr}^{-j}V$ equals the multiplicity of $j$ as a jump of $\mathrm{Jac}(C)$.
\end{theorem}
\begin{corollary} The image of $\vcan$ is $$\bigcup_{j\text{ jump of }\Jac\,{C}}\left(-j+\Z\right)$$ and the jumps of $\Jac\,{C}$ are rational numbers.
\end{corollary}

\begin{remark} By \cite[Corollary~5.5.7]{EHN15}, it follows that the smallest positive integer $e$ such that $\mathrm{Im}(\vcan)\subset \frac{1}{e}\Z$ is given by the so-called \emph{stabilisation index} $e(C)$ of $C$. It is not hard to see directly that $\mathrm{Im}(\vcan)\subset \frac{1}{e(C)}\Z$. On the other hand, the minimality of $e(C)$ does not seem to follow directly from our methods. On a similar note, our methods do not seem to yield the explicit formula for the jumps in terms of the combinatorial data of an \emph{snc} model \cite[Theorem~5.4.1]{EHN15}.
\end{remark}

\begin{question} Theorem~\ref{thm:MainThmcopy} suggests the following question, which we plan to study in forthcoming work.
Let $k$ be a Henselian discretely valued field with algebraically closed residue field $\tilde{k}$ and $A$ an Abelian $k$-variety. Let $\A$ be the N\'eron model of $A$ and equip the special fiber $\A_s$ with Edixhoven's descending filtration $\{F^\rho\A_s\}_{\rho\in[0,1)}$ by closed unipotent smooth subgroup $\tilde{k}$-schemes \cite[\S5.4.5]{Edi92}. Then does there exist a $\overline{\Q}$-valued discrete $k$-valuation $\vcan$ on $V\coloneqq \Lie(A)^\vee$ such that $\Lie(\A)^\vee=V^{>-1}$, which is the right analogue of the second equality in Lemma~\ref{lemma: Omega(C) and Omega log(C)}, and such that for all $j\in[0,1)$ there exist $\tilde{k}$-isomorphisms $\mathrm{Gr}^{-j}V\cong \mathrm{Gr}_j\A_s$? Here $\mathrm{Gr}^{-j}V$ denotes the $(-j)^{\text{th}}$ graded piece of the ascending filtration of $V$ induced by $\vcan$. An affirmative answer to the above question would imply rationality of Edixhoven's jumps for such $A$, which as mentioned above is an open problem \cite[\S1]{EHN15}. Although Theorem~\ref{thm:MainThmcopy} yields the case of Jacobians, we are not aware of a direct construction of isomorphisms $\mathrm{Gr}^{-j}V\cong \mathrm{Gr}_j\A_s$. 
\end{question}

\refstepcounter{equation}\subsection{Orthogonal canonical forms.} The proof of Theorem~\ref{thm:MainThmcopy} is obtained by studying orthogonal bases with respect to $\vcan$ and their behaviour under tamely ramified base change. Recall that a basis $\{\omega_i\}_{i\in I}$ of $V$ is called \emph{orthogonal} (with respect to $\vcan$) if for all $a_i\in k$ we have 
$$\vcan\left(\sum_{i\in I}a_i\omega_i\right) = \min_{i\in I}\vcan(a_i\omega_i).$$ 
If $k$ is complete, then $V$ admits an orthogonal basis \cite[Corollary~2.4.3.11]{BGR}. Modulo some details that we defer to Section~\ref{sec:ProofOfMainThm}, this reduces Theorem~\ref{thm:MainThmcopy} to Proposition~\ref{prop:main} below. 
\begin{proposition} \label{prop:main}
Assumptions as in Theorem~\ref{thm:MainThmcopy}. Suppose there exists an orthogonal basis $\{\omega_i\}_{i=1,\dots,g}$ of $V$ with respect to $\vcan$ so that $-1<\vcan(\omega_i)\le 0$ for all $i=1,\dots,g$. Then, up to rearranging, the tuple of jumps of $\mathrm{Jac}(C)$ is given by $$(-\vcan(\omega_1),\dots,-\vcan(\omega_g)).$$ 
\end{proposition}

\refstepcounter{equation}\subsection{A brief outline of the proof.} The proof of Proposition~\ref{prop:main} will be given in Section~\ref{sec:ProofOfMainThm}. We will first relate the jumps of $\mathrm{Jac}(C)$ to the change of the $R$-lattice $V^{>-1}$ under tame base change. Afterwards the main idea is that an orthogonal basis with respect to $\vcan$ is still orthogonal after a finite tamely ramified base field extension $k'/k$ for which the degree $[k':k]$ is coprime to the integer $e$ as in Section~\ref{intro: vcan is k-valuation}. For this, we require two additional key properties of $\vcan$:
\begin{enumerate}[label=(\alph*)]
\item the valuation $\vcan$ is invariant under tamely ramified base change, cfr. Lemma~\ref{lemma: props wt functions};
  \item for any $\omega\in V$, we have $\omega\in V^{>-1}$ if and only if $\omega$ extends to some holomorphic form on a (equivalently any) regular $R$-model $\C$ of $C$, cfr. Lemma~\ref{lemma: Omega(C) and Omega log(C)}.
\end{enumerate}
 
\begin{example}
   \label{example} 
   Let $k=\Q_2^{\mathrm{ur}}$ and consider the genus 2 hyperelliptic curve with affine equation $C\colon y^2 = 8x^6 + x^3 + 2$. This is a $\Delta_v$-regular curve in the sense of \cite{Dok21} -- see also Section~\ref{subs:delta-regular curves}. The algorithms in loc.\ cit.\ yield an explicit \emph{snc} model $\C$ of $C$ from the subdivided Newton polygon $\Delta$ of the affine equation, decorated with $v$-values on lattice points, see Figure~\ref{fig: example}.

\refstepcounter{equation}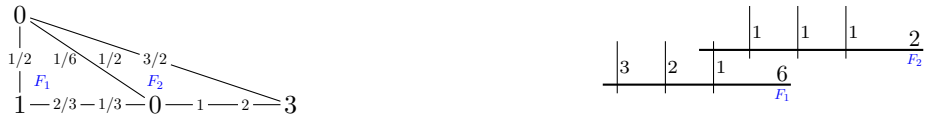
\begin{figure}[ht!]
   \input{deltareg_example}
   \caption{Left: the subdivided Newton polygon $\Delta$ for the given affine equation of $C$, decorated with a $v$-function whose values on lattice points are indicated. Right: special fiber of the minimal \emph{snc} model $\C$ of $C$.}\label{fig: example}
    \end{figure}
    In fact, the \emph{snc} model $\C$ obtained is minimal, and by \cite[Theorem~8.12]{Dok21}, there is an explicit $R$-basis $B=\{\omega_{(1,1)},\omega_{(1,2)}\}$ for $H^0(\C,\omega_{\C/R}),$ viewed as an $R$-submodule of $V$, so that on the affine part we have
   \[    \omega_{(1,1)} = \frac{\mathrm{d} x}{2y},\quad  \omega_{(1,2)} = \frac{\mathrm d x}{2}.\]
   One then computes $\vcan(\omega_{(1,1)}) = - \frac16$ and $\vcan(\omega_{(1,2)})=-\frac12$. 
Note that for any $a,b\in k$ we have $\vcan(a\omega_{(1,1)}+b\omega_{(1,2)})=\min\{v(a)-\frac16,v(b)-\frac12\}$, so $B$ is orthogonal with respect to $\vcan$. It follows from Proposition~\ref{prop:main} that the jumps of $C$ are given by $\frac12$ and $\frac16$. It is no coincidence that these are the $v$-values on the interior lattice points of $\Delta$ -- compare Theorem~\ref{cor:jumps delta-regular}. 
\end{example}
\begin{remark} In \cite{DDMM23} and subsequent works, it is shown how in odd residue characteristic much of the arithmetic of hyperelliptic curves is determined by the combinatorics of \emph{cluster pictures}. For instance, a basis of integral differential forms for certain such curves is explicitly constructed in \cite{Kunzweiler20} and \cite{Muselli22}. In fact, in both of these papers, the construction of $\vcan$ and orthogonal bases appear implicitly. For instance, the differential forms $\mu_i$ in \cite[Theorem~6.3]{Muselli22} as well as in \cite[Theorem~1.1]{Kunzweiler20} form an orthogonal basis with respect to $\vcan$, and in both papers the invariants $e_i$ satisfy $-\vcan(\mu_i)=[e_i]$, where $[e_i] = e_i-\lfloor e_i\rfloor$ denotes the decimal part of $e_i$. See also Remark~\ref{rmk: jumps for hyperelliptic curves}.

\end{remark}
\refstepcounter{equation}\subsection{Overview of the paper.} Section~\ref{sec:Preliminaries} contains some preliminaries on jumps, weight functions and $\vcan$. A detailed proof of Theorem~\ref{thm:MainThmcopy} can be found in Section~\ref{sec:ProofOfMainThm}. In Section~\ref{subs:delta-regular curves}, we make our constructions explicit in the case of $\Delta_v$-regular curves, generalising Example~\ref{example}.

\section{Preliminaries}
\label{sec:Preliminaries}

\refstepcounter{equation}\subsection{Notation and assumptions.} \label{subsection: setup}
Throughout the remainder of the article we keep the following assumptions. We let $k$ denote a Henselian discretely valued field with algebraically closed residue field $\tilde{k}$ of characteristic $p\geq 0$ and ring of integers $R$. For example, $k$ could be the unramified completion of a $p$-adic field. Fix an algebraic closure $k^a$ of $k$.

Let $\pi$ be a uniformiser of $R$ and we write $v$ for the normalised valuation of $k$, so $v(\pi)=1$. Since $k$ is assumed to be Henselian, for every finite seperable extension $k'/k$, there exists a unique normalised valuation $v'$ on $k'$ satisfying $v'(a)=[k':k]\cdot v(a)$ for all $a\in k$. We say $v'$ prolongs $v$. 

Let $C$ be a smooth, proper and geometrically connected $k$-curve of genus $g\ge 1$. We assume that $C$ has index one, i.e., admits a divisor of degree one. This condition is, for instance, met if $C$ contains a $k$-rational point. 
We write $J = \Jac(C)$ for the Jacobian variety of $C$ and $\omega_{C/k}$ for the canonical sheaf of $C$. Recall that $J$ is an Abelian $k$-variety and that the dual of the Lie algebra of $J$ is canonically isomorphic to the $k$-vector space $V\coloneqq H^0(C,\omega_{C/k})$ of canonical forms on $C$. 

If $X$ and $S'$ are schemes over a base scheme $S$, then we write $X_{S'}=X\otimes_S S'$. If $S'$ is the spectrum of a ring $R'$, we will also write $X_{R'}=X_{S'}$.

\refstepcounter{equation}\subsection{Jumps of Abelian varieties.}
\label{subsection: jumps}
Suppose $k'/k$ is a finite tamely ramified extension, let $R'/R$ be the associated extension of rings of integers and set $S=\Spec R$ and $S'=\Spec R'$. 
We let $\pi'$ be a uniformiser of $R'$. Let $A$ be an Abelian $k$-variety. 

Write $\A/S$ and $\A'/S'$ for the N\'eron models of $A/k$ and $A_{k'}/k'$, respectively. By the N\'eron mapping property, there exists a canonical morphism of smooth commutative $S'$-group schemes 
\begin{equation}\label{h}
h_{k'/k}\colon\A_{S'}\longrightarrow \A'.
\end{equation}
Consider the morphism on differentials induced by (\ref{h}), $$h^*_{k'/k}\Omega^1_{\A'/S'}\longrightarrow \Omega^1_{\A_{S'}/S'}=\Omega^1_{\A/S}\otimes_{\mathcal{O}_S} \mathcal{O}_{S'},$$ 
and apply the pullback along the unit section $e_{\A_{S'}/S'}\colon S'\to \A_{S'}$ to obtain the induced map of invariant differentials associated to (\ref{h}), denoted by
\begin{equation}
\label{dual lie(h)} \phi_{k'/k}\colon\omega_{\A'/S'}^{\mathrm{inv}}\to \omega_{\A/S}^{\mathrm{inv}}\otimes_R R',
\end{equation}
which is the dual of $\mathrm{Lie}(h_{k'/k})\colon \mathrm{Lie}(\A)\otimes_R R'\to \mathrm{Lie}(\A')$. Since $h_{k'/k}$ is an isomorphism on the generic fiber, it follows that (\ref{dual lie(h)}) is injective, and since $\phi_{k'/k}$ is a morphism of free $R'$-modules of the same rank $g$, the cokernel has finite length. We write $(c_i(A,k'/k))_{i=1,\dots,g}$ for the unique $g$-tuple of integers given by the normalised $k'$-valuations of the elementary divisors of $\mathrm{coker}(\phi_{k'/k})$ in nondecreasing order, that is $$\mathrm{coker}(\phi_{k'/k})\cong \bigoplus_{i=1,\dots,g} \frac{R'}{\pi'^{c_i(A,k'/k)}R'},$$ with $c_1(A,k'/k)\le \dots\le c_g(A,k'/k)$, and this characterises the integers $c_i(A,k'/k)$.

\begin{definition}[$(k'/k)$-jumps]
With notation as above, we define the \emph{$(k'/k)$-jumps} of $A$ to be the $g$-tuple of rational numbers 
$$j_i(A,k'/k) = \frac{c_i(A,k'/k)}{[k':k]}, \ \ \ i=1,\dots,g.$$ 
\end{definition}

We refer to \cite[\S6.1.3.7]{HN16} for the property that $j_i(A,k'/k)\in [0,1)$ for all $k'/k$ tame, and the property that $\phi_{k'/k}$ is an isomorphism if and only if $A$ has \emph{semi-abelian}\footnote{Recall that we say that an Abelian $k$-variety $A$ has semi-abelian reduction if the identity component of the special fiber of the N\'eron model of $A$ is a semi-abelian variety.} reduction. 

\begin{definition}[$k$-jumps] \label{def: k-jumps}
We define the \emph{$k$-jumps} of $A/k$ to be the $g$-tuple of real numbers 
\begin{equation}\label{eqn defn jumps}
j_i(A,k) = \lim_{k'/k\text{ tame}}j_i(A,k'/k), \ \ \ i=1,\dots,g 
\end{equation}
where the limit is taken over all tame extensions $k'/k$ in $k^a$ ordered by inclusion. 
\end{definition}

To see that the limit converges, first note that if $k''/k'/k$ is a tower of tamely ramified extensions, one has $$j_i(A,k''/k)=j_i(A,k'/k)+\frac{j_i(A_{k'},k''/k')}{[k':k]}.$$ Therefore, if we pick an infinite tower $k\subset k_1 \subset k_2 \subset \cdots$ of tamely ramified extensions which is cofinal among all tame extensions in $k^a$ ordered by inclusion, the resulting $j_i(A,k_n/k)$ form a bounded monotonically increasing sequence, and their limit in Equation~\eqref{eqn defn jumps} is thus well-defined.

\refstepcounter{equation}\subsection{Original definition of jumps.}\label{more on jumps and filtrations}
We refer to \cite[Section~6.1]{HN16} for a detailed comparison of Definition~\ref{def: k-jumps} and Edixhoven's original definition in \cite{Edi92}. In summary, Edixhoven constructed a descending filtration $F = \{F^\rho\A_s\}_{\rho\in[0,1) \cap \Z_{(p)}}$ by closed smooth $\tilde{k}$-subgroups, so that $\mathrm{Gr}^{\rho}\A_s=F^{<\rho}\A_s/F^{>\rho}\A_s$ has dimension equal to the number of $k$-jumps equal to $\rho$, also called the \emph{multiplicity} of $\rho$ as a jump of $A$. The filtration $F$ is obtained as a ``union'' of filtrations as follows: for every $e\in\Z_{\ge 1}\setminus p\Z$ (that is, $e$ is assumed coprime to $p$ in case $p> 0$) there exists a descending filtration of $\A_s$ by closed smooth $\tilde{k}$-subgroups $$\A_s=F^0_e\A_s\supset F^1_e\A_s\supset \dots\supset F^e_e\A_s=0$$ such that $F^a_e\A_s=F^{(a/e)}\A_s$. In particular, the filtration $\{F_e^a\}_{a=0,\dots,e-1}$ jumps at the integers $c_i(A,k'/k)$, and 
\begin{equation}
    j_i(A,k'/k)=\frac{\lfloor [k':k]\cdot j_i(A,k)\rfloor}{[k':k]},\label{eqn: k'/k-jumps in terms of k-jumps}
\end{equation}
where $k'$ is the unique degree $e $ extension of $k$ inside $k^a$.

\begin{lemma}\label{lem:jumps limit coprime} Fix any tower of tamely ramified extensions $k\subset k_1\subset k_2\subset \cdots$, where all the inclusions are strict. Then the limit in Equation~\eqref{eqn defn jumps} may as well be computed by only taking the limit over $(k_n/k)$-jumps.
\end{lemma}
\begin{proof} This follows from Equation~\eqref{eqn: k'/k-jumps in terms of k-jumps}.
\end{proof}

Note that in \cite[Definition~6.1.2.2]{HN16}, the definition of the integers $c_i(A,k'/k)$ may seem to be different than ours, but they amount to the same (up to rescaling) as in fact $\mathrm{coker}(\phi_{k'/k})\cong \mathrm{coker}(\mathrm{Lie}(h_{k'/k}))$  by  \cite[\S~6.2.1.3]{HN16} -- for convenience of the reader, we recall the argument in the following lemma.

\begin{lemma} \label{lemma: injecitve morphism of finte rank free modules and coker of dual}
Suppose $f\colon M\to N$ is an injective morphism of free $R$-modules of the same finite rank, and write $f^\vee\colon N^\vee\to M^\vee$ for the dual morphism. Then $\mathrm{coker}(f)\cong \mathrm{coker}(f^\vee)$.
\end{lemma}

\begin{proof}
 We have the exact sequence $$0\lra \coker(f)^\vee \lra N^\vee \lra M^\vee \lra \mathrm{Ext}^1(\coker(f),R)\lra \mathrm{Ext}^1(N,R)=0$$ and since $f$ is assumed injective, $\coker(f)$ is torsion. This implies $\coker(f)^\vee=0$ as well as $\mathrm{Ext}^1(\coker(f),R)=\coker(f)$ and the claim follows.
 \end{proof}

\begin{remark} 
In a similar way, one can define jumps for any semi-Abelian $k$-variety $A/k$; see, for instance, \cite{HN16,O} and the references therein. 

Note that from the definition it is not clear that the jumps are rational numbers, although as mentioned in Section~\ref{motivation}, this is known in our case of interest, namely when $A=\Jac(C)$. The jumps are also known to be rational in the case where $A$ is a tamely ramified\footnote{Meaning $A$ admits semi-abelian reduction after a tamely ramified base field extension.} abelian variety -- in this case, by \cite[6.1.2.4]{HN16}, the $k$-jumps of $A$ equal the $(k'/k)$-jumps of $A$, where $k'/k$ is any extension so that $A_{k'}$ has semi-abelian reduction. Moreover, the limit in Equation~\eqref{eqn defn jumps} stabilises if and only if $A$ is tamely ramified. 
\end{remark}

\refstepcounter{equation}\subsection{Weight functions.}
\label{subsection: weight functions}

Suppose $\C$ is an \emph{snc} $R$-model of $C/k$, as defined in Section~\ref{sec:CanVal}. By Lipman's resolution of singularities of 2-dimensional schemes, it is well-known that $C$ admits an \emph{snc} $R$-model; see, for instance, \cite[\S8.3.51 and \S9.3.36]{Liu02}. Let us write $S=\mathrm{Spec}\,R$ and $\omega_{\C/S}$ for the relative canonical sheaf of $\C/S$.

We equip $\C$ and $S$ with the standard logarithmic structure, which is the divisorial logarithmic structure given by the special fiber, and we denote by $\C^+$ and $S^+$ the resulting logarithmic schemes. As $\C$ is regular, the sheaf of logarithmic K\"ahler differentials $\Omega_{\C^+/S^+}$ is perfect, and so we can consider the associated determinant line bundle.

\begin{definition}[Logarithmic canonical sheaf]
\label{def:LogCanShf} 
The \emph{logarithmic canonical sheaf} $\omega^{\log}_{\C/S}$ is defined as the determinant sheaf of $\Omega_{\C^+/S^+}^{\log}$. 
\end{definition} 

By \cite[\S3.3.1(1)]{JN20}, we have that 
\begin{equation}
    \label{eqn: log can sheaf and can sheaf} \omega_{\C/S}^{\log}=\omega_{\C/S}(\C_{s,\red}-\C_s)
\end{equation}
as subsheaves of $\iota_*\omega_{C/k}$, where $\iota:C\to\C$ denotes the open immersion. Here, we use that $k$ has perfect residue field.

\begin{definition}[Weight functions] Suppose $k$ is complete. \label{def:wt} 
Suppose $\omega\in V$ is a nonzero canonical form, and let $C^{\an}$ denote the Berkovich analytification of $C$. Then following \cite{MN15}, there is a \emph{weight function} $$\wt_{\omega}\colon C^{\an}\to \overline{\R}\coloneqq \R\cup\{\infty\}$$ attached to $\omega$, characterised as the unique piecewise affine function satisfying the following property.  
Suppose $\C/S$ is any \emph{snc} model of $C/k$. Suppose $E$ is an irreducible component of the special fiber $\C_s$ and denote by $x\in C^{\an}$ the divisorial point associated to the divisorial valuation $$v_E(\cdot)\coloneqq \frac{1}{\mathrm{mult}(E)}\ord_E(\cdot).$$ Write $\sigma$ for the rational section of the logarithmic canonical sheaf $\omega^{\log}_{\C/R}$ induced by $\omega$, then $$\wt_\omega(x)=v_E(\div(\sigma)).$$
\end{definition}

Note that this definition of the weight functions is off by an additive constant of $1$ from the definition in \cite{MN15}, and we follow the conventions of \cite{JN20} -- see \cite[Remark~2.2.1]{JN20}.

\begin{lemma} \label{lemma: props wt functions}
Suppose $k$ is complete and let $\omega\in V$ be nonzero. 
\begin{enumerate}[label=(\alph*)]
    \item \label{lemma: props wt functions_1} The weight function $\wt_{\omega}$ is bounded below and attains its minimum in a divisorial point associated to a divisorial valuation of the form $v_E$ as in Definition~\ref{def:wt}, for any choice of \emph{snc} model $\C$ of $C$. In particular, this minimal value is a rational number.
\item For any finite tame extension $k'/k$ in $k^a$, we have $$\wt_{\omega_{k'}}(x)=[k':k]\cdot\wt_{\omega}(\rho(x))$$ for all $x\in C^{\an}_{k'}$ where $\rho\colon  C^{\an}_{k'} \to C^{\an}$ is the projection.
\end{enumerate}
\end{lemma}
\begin{proof}
For (a), see for instance \cite[Theorem~4.7.5]{MN15}, and for (b) see for instance \cite[\S4.5.1(2)]{JN20}.
\end{proof}

\refstepcounter{equation}\subsection{A canonical valuation on $H^0(C,\omega_{C/ k})$.} \label{subsection: canonical valuation}
\begin{definition}\label{defn: vcan}
    Let $\C$ be an \emph{snc} model of C and write $\{E_i\}_{i \in I}$ for the set of irreducible components of $\C_{s,\red}$.
    Let $\omega \in H^0(C, \omega_C) \setminus \{0\}$.
    Let $\sigma$ be the rational section of $\omega_{\C / R}^{\log}$ induced by $\omega$. 
    Then we define \[
        \vcan(\omega) = \min_{i \in I}v_{E_i} (\div \sigma)
    .\]
    By convention, $\vcan(0) = \infty$. We refer to $v_\can$ as the \textit{canonical valuation}.
\end{definition}

A few words about $\vcan$ are in order. First, this definition coincides with the one given in Section~\ref{sec:CanVal}, as a consequence of the isomorphism in \eqref{eqn: log can sheaf and can sheaf}; see also the proof of Lemma~\ref{lemma: Omega(C) and Omega log(C)}. Second, $\vcan$ is a discrete $k$-valuation on $V$ (Section~\ref{intro: vcan is k-valuation}), as it can be expressed as a minimum of finitely many discrete valuations $v_{E_i}$. 

\begin{lemma}\label{lemma: defn vcan}
    Definition~\ref{defn: vcan} is independent of the chosen \emph{snc} model $\C$ of $C$. More specifically, for all nonzero $\omega \in H^0(C, \omega_C)$, we have
    \[
        \vcan(\omega) = \min_{x \in C_{\widehat k}^{\mathrm{an}}} \wt_{\omega \otimes 1}(x),
    \]
    where $\hat{k}$ denotes the completion of $k$.
\end{lemma}
\begin{proof} It suffices to prove the second statement. Note that if $\C / R$ is an \emph{snc} $R$-model of $C$, then by \cite[Lemma~8.3.49.(a)-(b)]{Liu02}, $\C_{\hat{R}}$ is an \emph{snc} $\hat{R}$-model of $C_{\hat{k}}$ with unchanged special fiber, and formation of relative canonical sheaves commutes with flat base change. It follows that $\vcan$ can be computed after base change to the completion, and in that case we can apply the defining property of $\wt_{\omega}$ in Definition~\ref{def:wt} and \cite[Proposition~4.5.5]{MN15}.
\end{proof}

\begin{definition}
\label{def:OmegaC}
We define 
$$\Omega^{\log}(C)\coloneqq V^{\ge 0}\ \ \ \text{and}\ \ \ \Omega(C)\coloneqq V^{>-1},$$ 
where recall that for any $\rho\in\R$ we write $V^{\ge \rho}=\{\omega\in V:\ \vcan(\omega)\ge \rho\}$, and similarly for $V^{>\rho}$.
\end{definition}

\begin{lemma} 
\label{lemma: Omega(C) and Omega log(C)}
Let $\C$ be an \emph{snc} $R$-model of $C$. Viewing $H^0(\C, \omega_{\C / R}^{\mathrm{log}})$ and $ H^0(\C, \omega_{\C / R})$ as $R$-submodules of $V$, then as sets we have
$$H^0(\C, \omega_{\C / R}^{\mathrm{log}})=\Omega^{\log}(C) 
\ \ \ \text{and} \ \ \ 
H^0(\C, \omega_{\C / R}) = \Omega(C).$$
In particular, these $R$-submodules of $V$ are independent of the choice of \emph{snc} model $\C$.
\end{lemma}

\begin{proof} Note that since $\omega_{\C/R}$ and $\omega_{\C/R}^{\log}$ are line bundles, the $R$-modules $H^0(\C, \omega_{\C / R}^{\mathrm{log}})$ and $H^0(\C, \omega_{\C / R})$ are torsion-free, and so we these $R$-modules can indeed be viewed as $R$-submodules of $H^0(C,\omega_{C/k})=V$.

    Suppose $\omega \in H^0(C, \omega_C)$. Let $\rho$ be the rational section of  $\omega_{\C/R}$ induced by $\omega$, and let $\sigma$ be the rational section of  $\omega^{\mathrm{log}}_{\C/R}$ induced by $\omega$.
    As $\C$ is normal, $\rho$ (resp.\ $\sigma$) extend to regular sections on $\C$ if and only if for every irreducible component $E$ of $\C_{s,\red}$, the integers $\ord_E(\div\rho)$ (resp.\ $\ord_E(\div\sigma)$) are nonnegative. The first equality then follows from Lemma~\ref{lemma: props wt functions}\ref{lemma: props wt functions_1}. For the second equality, we can use Equation~\eqref{eqn: log can sheaf and can sheaf} to see that in a neighbourhood of the generic point of $E$, we have $$\div\rho - (\mathrm{mult}(E)-1)\cdot E = \div\sigma,$$ which gives 
    $$v_E(\div\rho)-1+\frac{1}{\mathrm{mult}(E)}=v_E(\div\sigma).$$ 
    This yields $\vcan(\omega)>-1$ if and only if for all irreducible components $E$ of $\C_s$, it holds that $v_E(\div\rho)\ge 0$.
\end{proof}
\begin{lemma} \label{lemma: tangent space via regular models} Suppose $k'/k$ is any finite extension, and denote by $R'/R$ the associated extension of rings of integers. View $\Omega(C)\otimes R'$ and $\Omega'(C)\coloneqq \Omega(C_{k'})$ as sublattices of $V'\coloneqq H^{0}(C_{k'},\omega_{C_{k'}/k'})\cong V\otimes k'$. We then have $$\Omega'(C)\subset \Omega(C)\otimes R',$$ and the inclusion is identified with the map $\phi_{k'/k}$ in Equation~\eqref{dual lie(h)}.
\end{lemma}
\begin{proof} This follows from \cite[Proposition~2.4.4]{EHN15}: we can view $\omega^{\text{inv}}_{\J/S}$ as an $R$-submodule of $V$, and doing so we have $\omega^{\text{inv}}_{\J/S} = \Omega(C)$ -- note that here we use the assumption that $C$ has index one.
\end{proof}

\section{Proof of Theorem~\ref{thm:MainThmcopy}}\label{sec:ProofOfMainThm}

In this section, we prove the main theorem. The idea is to find an $R$-basis of $\Omega(C)$ which simultaneously diagonalizes the inclusions in Lemma~\ref{lemma: tangent space via regular models} for all tame extensions $k'/k$ of degree coprime to some integer. 
To this end, we can use a basis of $V$ which is \emph{orthogonal} with respect to $\vcan$.
We show that a suitably normalized orthogonal basis of $V$ becomes an $R$-basis of $\Omega(C)$ which remains orthogonal after base change to $k'$. 
We then recover the jumps of $C$ via the scaling factors required for renormalizing to an $R'$-basis for $\Omega(C')$. 

\begin{lemma}\label{lem:unique_extension_val_vector_space}
Let $W$ be a finite dimensional $k$-vector space equipped with a discrete $k$-valuation $\lambda\colon W \to \frac{1}{e}\overline{\Z}$ for some positive integer $e$.
 Let $k'/k$ be a finite tame extension of degree $d$. We assume that $d$ is coprime to $e$.
\begin{enumerate}[label=(\alph*)]
\item \label{lem:unique_extension_val_vector_space_a} There is a unique discrete $k'$-valuation $\lambda'$ on $W' = W \otimes_k k'$ so that for all $w\in W$ we have $$\lambda'(w\otimes 1)=d\cdot \lambda(w)$$ We say $\lambda'$ \emph{prolongs} $\lambda$.

 \item \label{lem:unique_extension_val_vector_space_b} If $\{w_i\}_{i\in I}$ is a orthogonal basis of $W$ with respect to $\lambda$, then $\{w_i \otimes 1\}_{i\in I}$ is an orthogonal basis for $W'$ with respect to the unique prolongation $\lambda'$ of $\lambda$.  
 
 \item \label{lem:unique_extension_val_vector_space_c} Let $\{w_i\}_{i\in I}$ be an orthogonal basis of $W$. Assume $-1<\lambda(w_i)\le 0$ for all $i\in I$. Then $\{w_i\}_{i\in I}$ is an $R$-basis of $W^{\lambda>-1}$.
 
\end{enumerate}
\end{lemma}

\begin{proof} For (a), choose a uniformiser $\pi'$ of $k'$. Then $\{1,\dots,\pi'^{d-1}\}$ is a $k$-basis for $k'$ and so any valuation $\lambda'$ on $W'=\bigoplus_{i = 0}^{d-1}{\pi'}^{i}W$ prolonging $\lambda$ must satisfy $$\lambda'(\pi'^i(w\otimes 1))\in i +\frac{d}{e}\Z$$ for all $w\in W$ and $0\le i<d$. By assumption, $d$ and $e$ are coprime and so the sets $i+\frac{d}{e}\Z$ have pairwise empty intersection for $0\le i<d$. Therefore, we must have 
$$\lambda'\left(\sum_{0\le i<d} \pi'^i(w\otimes 1)\right)=\min_{0\le i<d}\lambda'(\pi'^i(w\otimes 1)),$$
and this defines a unique $k'$-valuation prolonging $\lambda$.

For (b), it follows from the definition of orthogonality that there exists a unique prolongation of $\lambda$ to $W'$ for which $\{w_i\otimes 1\}$ remains orthogonal. By the unicity in (a), we see that in fact $\lambda'$ must be this prolongation of $\lambda$.

For (c), suppose that $w=\sum_{i\in I} a_iw_i\in W^{\lambda>-1}$ for some $a_i\in k$.
We need to show that $a_i\in R$. To do so, note that $$-1<\lambda(w)=\min_{i\in I}v(a_i)+\lambda(w_i)\le \min_{i\in I} v(a_i),$$ so that indeed $v(a_i)\ge 0$ for all $i\in I$.
\end{proof}

\begin{proof}[Proof of Proposition~\ref{prop:main}] 
    Let $B=\{\omega_1,\dots,\omega_g\}$ be an orthogonal basis of $V$ with respect to $\vcan$ such that $-1<\vcan(\omega)\le 0$ for all $i$. By Lemma~\ref{lem:unique_extension_val_vector_space}\ref{lem:unique_extension_val_vector_space_c}, $B$ is also an $R$-basis of $\Omega(C)$.

    Let $e$ be any integer so that $\mathrm{Im}(\vcan)\subset \frac{1}{e}\Z$. Let $k'/k$ be a finite tame extension in $k^a$ of degree coprime to $e$, and denote by $R'/R$ the associated extension of rings of integers. For each such $k'/k$, write $\vcan'$ for the canonical valuation on $H^0(C_{k'}, \omega_{C_k'})$.
    By \cref{lemma: props wt functions}, we see $$\vcan'(\omega \otimes 1)=[k':k]\cdot\vcan(\omega)$$ for every $\omega \in H^0(C, \omega_C)$. 
  In other words $\vcan'$ prolongs $\vcan$. 
    By \cref{lem:unique_extension_val_vector_space}\ref{lem:unique_extension_val_vector_space_b}, $\{\omega_1,\dots,\omega_g\}$ is again an orthogonal basis of $H^0(C_{k'}, \omega_{C_{k'}/ k'})$ with respect to $\vcan'$.
    By Lemma~\ref{lem:unique_extension_val_vector_space}\ref{lem:unique_extension_val_vector_space_c}, we therefore find that 
    $$\left\{\pi'^{\lfloor -[k':k]\vcan(\omega_i)\rfloor} \omega_i\right\}_{i=1,\dots,g}$$ is an $R'$-basis of $\Omega(C_{k'})$.
    Now, using Lemma~\ref{lemma: tangent space via regular models}, we see that the $(k'/k)$-jumps of $C$ are given by \begin{equation}\label{eqn: end proof}
j_i(\Jac(C),k'/k) = \frac{\lfloor -[k':k]\cdot \vcan(\omega_i)\rfloor}{[k':k]}.
\end{equation}

Finally, by Lemma~\ref{lem:jumps limit coprime}, to compute the $k$-jumps it suffices to take the limit of $(k'/k)$-jumps over all field extensions $k'/k$ as above, and from the Equation~\eqref{eqn: end proof} we see that this limit equals $-\vcan(\omega_i)$.  
\end{proof}  

We are now ready to prove the main result.

\begin{proof}[Proof of Theorem~\ref{thm:MainThmcopy}] 
First note that if there exists an orthogonal basis with respect to $\vcan$, then we are done by Proposition~\ref{prop:main}. In particular, this is the case if $k$ is complete, as then any finite dimensional $k$-vector space is Banach, and any Banach space over a discretely valued base field admits an orthogonal basis \cite[Corollary~2.4.3.11]{BGR}. 

In the proof of Lemma~\ref{lemma: defn vcan}, we saw that completion respects \emph{snc} models and their relative canonical sheaves, so the completion of $V$ with respect to $\vcan$ is naturally isomorphic to $H^0(C_{\hat{k}},\omega_{C_{\hat{k}}/\hat{k}})$ equipped with its canonical valuation, and the value group is preserved under completion. So let us check that jumps are invariant under completing the ground field. Suppose $k'/k$ is a finite tamely ramified extension. 
By Lemma~\ref{lemma: tangent space via regular models}, we can identify $\phi_{k'/k}$ with the inclusion $\Omega'(C)\subset \Omega(C)\otimes_R R'$. One sees that $\phi_{\hat{k'}/\hat{k}}$ is the completion of $\phi_{k'/k}$, and by flatness of completion with respect to a discrete valuation, we then find $\widehat{\coker(\phi_{k'/k})}=\coker(\phi_{\hat{k'}/\hat{k}})$, and so $c_i(C, k' / k) = c_i (C_{\hat k}, \widehat{k'} / \hat k)$ for all $i=1,\dots,g$. 
\end{proof}

\section{Jumps for $\Delta_v$-regular curves} \label{subs:delta-regular curves} 

In this section, we illustrate how to use our main result to compute the jumps efficiently for so-called $\Delta_v$-regular curves, as studied by Dokchitser in \cite{Dok21}. 
These are curves $C/k$ for which there is an affine part given by an equation $C_0:f=0$ for some $f\in k[x,y]$ so that the Newton polygon $\Delta$ of $f$, decorated with a piecewise affine-linear function $\Delta\to \R$, contains enough data to construct an \emph{snc} model $\C_\Delta$ of $C$ via toroidal embeddings. In \cite{Dok21}, this piecewise affine-linear function is denoted by $v$; in order to avoid any confusion, we will henceforth write $v_k$ for the valuation on the field $k$. For the reader's convenience, we summarize some of the notions in loc.\ cit.\ below.
For later use, we first recall a classical result on Newton polygons. 

\refstepcounter{equation}\subsection{Baker's theorem.} \label{subs: nondeg} In this section, let $k$ be any field, not necessarily valued. Suppose that $C$ is a smooth, proper and connected $k$-curve, birational to an affine curve $C_0:f=0$ in the torus $\mathbb{G}_{m,k}^2$ for some $f=\sum_{i,j}a_{ij}x^iy^j\in k[x,y]$. Assume $f_y\coloneqq\frac{\partial f}{\partial y}\ne 0,$ or equivalently $k(C)/k(x)$ is a separable extension. Let 
    $$\Delta=\text{convex hull of }\{(i,j) \mid a_{ij}\ne 0\}\subset \R^2$$ 
    be the Newton polygon of $f$ over $k$. Write $\Delta^{\circ}$ for the interior of $\Delta$. For each lattice point $(i,j)\in \Z^2$, we denote by $\omega_{(i,j)}$ the unique meromorphic canonical form on $C$ which on the affine part is given by $x^{i-1}y^{j-1} \frac{dx}{f_y'}$.

    For every edge $L$ of $\partial\Delta$, denote $r_L=|L\cap\Z^2|-1$, fix an ordering $(i_0,j_0),\dots,(i_{r_L},j_{r_l})$ of $L\cap\Z^2$, and define $f_L$ as the degree $r_L$ polynomial $\sum_{n=0}^r a_{i_n,j_n}t^n\in k[t]$.  
    
    We say that $C$ is nondegenerate (or outer regular) at $L$ if $f_L,(xf_x)_L,(yf_y)_L$ have no common roots. 
    We say that $C$ is \emph{nondegenerate} with respect to $\Delta$ if $C$ is nondegenerate with respect to all edges of $\Delta$. Equivalently, $C$ is nondegenerate if and only if $C$ can be obtained as the closure of $C_0$ in the projective toric surface associated to $\Delta$, and in this case points of $C\setminus C_0$ correspond bijectively to the roots of $f_L$ where $L$ runs over the edges of $\partial\Delta$. In particular, \begin{equation}
        |C\setminus C_0|=\sum_{L\subset\partial\Delta} r_L=|\partial\Delta\cap \Z|\label{eqn:partial delta}.\end{equation}
    
    As before, we denote $V=H^0(C,\omega_{C/k})$. The well-known Baker's theorem states if $C$ is nondegenerate with respect to $\Delta$, then $$\{\omega_{(i,j)}\}_{(i,j)\in\Delta^{\circ}\cap\Z^2}$$ is a $k$-basis of $V$ -- see for instance \cite[\S2]{Dok21} for a modern proof.
    In particular, if $(i,j)$ is an interior lattice point of $\Delta$, then $\omega_{(i,j)}$ is holomorphic on $C$. If $(i,j)\in \Z^2$ is any other lattice point, then $\omega_{(i,j)}$ will acquire some poles at the points of $C\setminus C_0$. For later use, we prove the following lemma, which expands on \cite[Remark~2.3]{Dok21}.

    \begin{lemma}\label{lemma: presentation canonical forms with poles}
    The $k$-vector space $H^0\left(C,\omega_{C/k}(C\setminus C_0)\right)$ of holomorphic forms with at most simple poles away from $C_0$ admits a presentation of $k$-vector spaces
    \end{lemma}
    \begin{equation}\label{seq:presenation can forms}
        0\longrightarrow k\cdot\left(\sum_{i,j} a_{ij}\omega_{(i,j)}\right)\longrightarrow\bigoplus_{(i,j)\in\Delta\cap\Z^2}k\cdot\omega_{(i,j)}\longrightarrow H^0(C,\omega_{C/k}(C\setminus C_0))\longrightarrow0.
    \end{equation} 
    \begin{proof} Each of the forms $\omega_{(i,j)}$ for $(i,j)\in\Delta^\circ\cap\Z^2$ acquires at most simple poles away from $C_0$ by \cite[Theorem 2.2(4)]{Dok21}. Over $C_0$ the sheaf $\omega_{C/k}$ is generated by $\frac{dx}{f_y}$ as an $\mathcal{O}_C$-module, and so by the irreducibility of $f$ it follows that (\ref{seq:presenation can forms}) is left-exact. On the other hand, a Riemann--Roch computation and Equation~\eqref{eqn:partial delta} show that $\dim_kH^0(C,\omega_{C/k}(C\setminus C_0))=|\partial\Delta\cap\Z|-1$, and so (\ref{seq:presenation can forms}) is right-exact by a dimension count.
    \end{proof}

\refstepcounter{equation}\subsection{$\Delta_v$-regular curves.} \label{subs: delta v}

Keep the notation and assumptions of Section~\ref{subs: nondeg}. Additionally, assume that $k$ is equipped with a normalised discrete valuation $v_k:k\twoheadrightarrow \overline{\Z}$ and write $R$ for the ring of integers of $k$.

Let
$$\Delta_v=\text{lower convex hull of }\{(i,j,\val(a_{ij})) \mid a_{ij}\ne0\}\subset \R^2\times \R$$
be the Newton polytope of $f$ over $R$.
We write $v\colon\Delta\to \R$ for the unique piecewise affine-linear function satisfying $(p,v(p))\in \Delta_v$ for all $p\in\Delta$. The images of faces of the polytope $\Delta_v$ under the canonical projection $\Delta_v\to\Delta$ subdivide $\Delta$ into what we call $v$-faces. Alternatively, we can describe $v$-faces as the maximal simply connected subsets of $\Delta$ on which $v$ is affine-linear.

\begin{lemma}\label{lemma: vcan of delta-regular basis}
Assume $C$ is $\Delta_v$-regular, in the sense of \cite[Definition~3.9]{Dok21}.  Then
$$\vcan(\omega_{(i,j)}) = -v(i,j).$$
\end{lemma}

\begin{proof} Let $\C_{\Delta}$ be the \emph{snc} model associated to $\Delta$ \cite[Theorem~3.13]{Dok21}. Let $F$ be a $v$-face of $\Delta$ and write $x_F$ for the corresponding component of $(\C_\Delta)_s$. It suffices to prove the stronger claim that \begin{equation}
  v_{x_F}(\omega_{(i,j)})\ge -v(i,j) \label{eqn: wt of can integral basis for delta-regular curve}
\end{equation}
with equality if and only if $(i,j)\in F$; we remind the reader that $v_{x_F}(\cdot)$ is defined in (\ref{eqn:intro defn vE}).
By \cite[Proposition~8.1]{Dok21} it follows that $\ord_{x_F}\omega_{(i,j)}=F^*(i,j)-1+\delta_F$, where $\delta_F$ is the multiplicity of the component of $(\C_\Delta)_s$ corresponding to $F$, and $F^*\colon\Z^2\to\Z$ is the unique affine-linear function coinciding with $-\delta_Fv(\cdot)$ on $F\cap\Z^2$. Therefore $v_{x_F}(\omega_{(i,j)})=\frac{1}{\delta_F}F^*(i,j)$ and by convexity of $(i,j)\mapsto v(i,j)$, we have $F^*(i,j)\ge -\delta_Fv(i,j)$ with equality if and only if $(i,j)\in F$. \end{proof}

\begin{remark} It is worth remarking that if $k$ is complete then by 
Equation~\eqref{eqn: wt of can integral basis for delta-regular curve} it follows that the minimum locus of $\wt_{\omega_{(i,j)}}$, also known as the \textit{Kontsevich--Soibelman skeleton} attached to $\omega_{(i,j)}$ -- see for instance \cite{MN15}, coincides with the union of components $x_F$ associated to $v$-faces $F$ for which $(i,j)\in F$, as well as all edges between these components in the dual graph of $\C_{\Delta}$. This can be used to prove a result similar to \cite[Theorem~3.4.6]{BN} for $\Delta_v$-regular curves; details will appear in the second author's forthcoming PhD thesis.
\end{remark}

\begin{proposition} \label{prop: ortho basis delta regular}
    Suppose $C$ is a $\Delta_v$-regular curve. The basis $\{\omega_{(i,j)}\}_{(i,j)\in\Delta^\circ\cap \Z^2}$ is orthogonal with respect to $\vcan$.
\end{proposition} 

\begin{proof} The argument is more or less the same as that of \cite[Theorem~8.12]{Dok21}. Arguing by contradiction, we suppose that $\omega=\sum u_{ij}\omega_{(i,j)}$ is a nonzero canonical form for which $$\vcan(\omega)> \min \vcan(u_{ij}\omega_{(i,j)}).$$ We can assume the right hand side equals $m<\infty$ and that this minimum is attained on $x_F$ for some $v$-face $F$ of $\Delta$.

Denote $J_F=\{(i,j)\in \Delta^{\circ}\cap \Z^2 \mid \wt_{u_{ij}\omega_{(i,j)}}(x_F)=m\}$. By Equation~\eqref{eqn: wt of can integral basis for delta-regular curve}, it follows $J_F\subset F$. By assumption, $\wt_\omega(x_F)>m$ , and the scaled reduction of $\sum_{(i,j)\in J_F}u_{ij}\omega_{(i,j)}$ must therefore vanish on $x_F$. By Lemma~\ref{lemma: presentation canonical forms with poles} this can only hold if $J_F$ contains the vertices of $F$, as in the proof of \cite[Theorem~8.12]{Dok21}.
But now we can repeat the argument for neighbouring faces, and we find that for every face $F$, the set $J_F$ contains all the vertices of $F$, and thus every vertex of $\Delta$. This is a contradiction since all the $J_F$ only consist of interior lattice points.
\end{proof}

\begin{theorem} \label{cor:jumps delta-regular} 
Assumptions as in Section~\ref{subsection: setup}. Assume $C$ is $\Delta_v$-regular. Then the $g$-tuple of jumps of $\mathrm{Jac}(C)$ is given by $$\left( [v(i,j)]\right)_{(i,j)\in\Delta^{\circ}\cap\Z^2},$$ where $[x]$ denotes the decimal part of $x$. 
\end{theorem} 

\begin{proof} By Proposition~\ref{prop: ortho basis delta regular}, $$\{\pi^{\lfloor v(i,j)\rfloor} \omega_{(i,j)}\}_{(i,j)\in \Delta^\circ\cap\Z^2}$$ is an orthogonal basis of $V$, which is in fact an $R$-basis of $\Omega(C)$ by Lemma~\ref{lemma: vcan of delta-regular basis} and Lemma~\ref{lem:unique_extension_val_vector_space}\ref{lem:unique_extension_val_vector_space_c}. The conclusion follows from Proposition~\ref{prop:main}.
\end{proof}

Three remarks are in order.

\begin{remark} \label{rmk: extension to non-delta regular case}
    We expect that a suitable version of Theorem~\ref{cor:jumps delta-regular} is still true in the non-$\Delta_v$-regular case, where $v$-values of interior lattice points of $\Delta_v$-regular faces still contribute to jumps. A related empirical fact is that for some non-$\Delta_v$-regular $k$-curves, one can construct \emph{snc} models by glueing several $\Delta_v$-regular affine charts, as in \cite[\S10]{Dok21}.
\end{remark}

\begin{remark}\label{rmk: alternate proof jumps delta-regular}
Theorem~\ref{cor:jumps delta-regular} can alternatively be proved without invoking Proposition~\ref{prop: ortho basis delta regular} as follows. Let $e$ denote any positive integer for which $v(i, j) \in  \frac{1}{e}\Z$ for all $(i, j) \in \Delta\cap\Z^2$. Then \cite[Remark~3.11]{Dok21} implies that $C$ remains $\Delta_v$-regular after base change to a finite tame extension $k'/k$ for which $[k':k]$ is coprime to $e$. Moreover after such a base change the subdivided Newton polygon $\Delta_v$ remains the same and $v:\Delta_v\to \R$ and $v':\Delta_v'=\Delta_v\to\R$ are related by $v'(\cdot)=[k':k]\cdot v(\cdot)$. Now for any such $k'/k$, the inclusions $\Omega'(C)\subset \Omega(C)\otimes R'$ are diagonalised with respect to the $R'$-bases $\{\pi'^{\lfloor v'(i,j)\rfloor}\omega_{(i,j)}\}$ and $\{\pi^{\lfloor v(i,j)\rfloor}\omega_{(i,j)}\}$ respectively, where $(i,j)$ runs over $\Delta^{\circ}\cap\Z^2$. Now similarly one retrieves the jumps of $\Jac(C)$ as in the proof of Proposition~\ref{prop:main}.
\end{remark}
\begin{remark}    \label{rmk: jumps for hyperelliptic curves} 
A very similar approach as in Remark~\ref{rmk: alternate proof jumps delta-regular} allows to explicitize the jumps of hyperelliptic $k$-curves $C:y^2=f(x)$ in residue characteristic different from $2$ admitting a so-called \emph{almost rational cluster picture}, as introduced in \cite[Definition 1.1]{Muselli22}. More specifically, we claim, under these conditions and those of Section~\ref{subsection: setup}, that the tuple of jumps of $\Jac(C)$ is given by the decimal parts of the invariants $e_i$ defined in \cite[Theorem~6.3]{Muselli22}\footnote{Note the $y$-regular condition is automatically satisfied in residue characteristic not $2$.}; let us indicate an argument below.

Consider any tame extension $k'/k$ whose degree is coprime to the degree of the splitting field of $f$. The rational cluster picture of $C$ is invariant after base change to any such $k'/k$, and moreover the centers of the corresponding clusters can be taken the same. Let $\{\mu_1,\dots,\mu_g\}$ be the basis of integral differentials constructed in \cite[Theorem~6.3]{Muselli22}, where $g$ is the genus of $C$. The constructions there allow us to make the same choices of clusters $\mathfrak{s}_i$ in the definition of the $\mu_i$ after any such base change $k'/k$. With respect to these uniform choices, the associated invariants $e_i$ and $\mu_i$ satisfy $$e_i'=[k':k]e_i \quad \text{and} \quad \mu_i'=\pi'^{\lfloor e'_i\rfloor}\pi^{-\lfloor e_i\rfloor}\mu_i.$$ Now similarly as for $\Delta_v$-regular curves, for all $k'/k$ as above the inclusions $\Omega'(C)\subset \Omega(C)\otimes R'$ are diagonalised with respect to the bases $\{\mu_i'\}$ and $\{\mu_i\otimes 1\}$, and once more we can retrieve the jumps of $\Jac(C)$ as in the proof of Proposition~\ref{prop:main}.
\end{remark}

\subsection*{Acknowledgements.} We were led to our main result via Theorem~\ref{cor:jumps delta-regular}, which in turn was inspired from doing computations with the software packages of Tim Dokchitser for computing models of curves, see \cite{Dok21}; we thank him heartily for making these packages publicly available.
Moreover, we are grateful to Lars Halle, Johannes Nicaise and Ismaele Vanni for several interesting discussions and comments on a first draft. We also would like to thank the anonymous referees for several useful suggestions.

\printbibliography

\end{document}

%% file: deltareg_example.tex
\vskip 7pt
{\hfill\pbox[c]{10cm}{
\begin{tikzpicture}[xscale=0.6,yscale=0.6,
  sml/.style={scale=0.55},
  lrg/.style={scale=0.9,inner sep=0.2em},
  fname/.style={blue,scale=0.55},
  lin/.style={-,shorten <=-0.07em,shorten >=-0.07em}]
\node[fname] at (0.50,0.50) {$F_1$};
\node[fname] at (3.0,0.50) {$F_2$};
\node[lrg] at (0,0) (1) {1};
\node[sml] at (0,1) (2) {1/2};
\node[lrg] at (0,2) (3) {0};
\node[sml] at (1,0) (4) {2/3};
\node[sml] at (1,1) (5) {1/6};
\node[sml] at (2,0) (6) {1/3};
\node[sml] at (2,1) (7) {1/2};
\node[lrg] at (3,0) (8) {0};
\node[sml] at (3,1) (9) {3/2};
\node[sml] at (4,0) (10) {1};
\node[sml] at (5,0) (11) {2};
\node[lrg] at (6,0) (12) {3};
\draw[lin]
 (3) edge (9) (9) edge (12)
 (1) edge (4) (4) edge (6) (6) edge (8)
 (8) edge (10) (10) edge (11) (11) edge (12)
 (3) edge (8)
 (1) edge (2) (2) edge (3)
;
\end{tikzpicture}
}\hfill\hfill
\pbox[c]{10cm}{
\begin{tikzpicture}[xscale=0.8,yscale=0.7,
  l1/.style={shorten >=-1.3em,shorten <=-0.5em,thick},
  l2/.style={shorten >=-0.3em,shorten <=-0.3em},
  lfnt/.style={font=\tiny},
  rightl/.style={right=-3pt,lfnt},
  mainl/.style={scale=0.8,above left=-0.17em and -1.5em},
  facel/.style={scale=0.5,blue,below right=-0.5pt and 6pt}]
\draw[l1] (0.00,0.00)--(2.26,0.00) node[mainl] {6} node[facel] {$F_1$};
\draw[l2] (0.00,0.00)--node[rightl] {3} (0.00,0.66);
\draw[l2] (0.80,0.00)--node[rightl] {2} (0.80,0.66);
\draw[l1] (1.60,0.66)--(4.46,0.66) node[mainl] {2} node[facel] {$F_2$};
\draw[l2] (1.60,0.00)--node[rightl] {1} (1.60,0.66);
\draw[l2] (2.20,0.66)--node[rightl] {1} (2.20,1.33);
\draw[l2] (3.00,0.66)--node[rightl] {1} (3.00,1.33);
\draw[l2] (3.80,0.66)--node[rightl] {1} (3.80,1.33);
\end{tikzpicture}
}\hfill}

\bigskip